\documentclass{amsart}

\usepackage{latexsym}
\usepackage{amssymb}

\newtheorem{theorem}{Theorem}[section]
\newtheorem{lemma}[theorem]{Lemma}
\newtheorem{proposition}[theorem]{Proposition}
\newtheorem{corollary}[theorem]{Corollary}

\theoremstyle{definition}
\newtheorem{definition}[theorem]{Definition}
\newtheorem{example}[theorem]{Example}

\theoremstyle{remark}
\newtheorem{remark}[theorem]{Remark}

\numberwithin{equation}{section}

\newcommand{\C}{\mathbb{C}}

\newcommand{\E}{\mathbb{E}}

\newcommand{\al}{\alpha}

\newcommand{\lm}{\lambda}

\newcommand{\Nor}{\mathcal N}

\begin{document}

\title[Examples of mixing subalgebras of von Neumann algebras]{Examples of mixing subalgebras of von Neumann algebras and their normalizers}

\author{Paul Jolissaint}
\address{  Universit\'e de Neuch\^atel,
       Institut de Math\'emathiques,       
       Emile-Argand 11,
       CH-2000 Neuch\^atel, Switzerland}
       
\email{paul.jolissaint@unine.ch}
\thanks{To appear in the \textit{Bulletin of the Belgian Mathematical Society Simon Stevin}}


\date{\today}


\begin{abstract}
We discuss different mixing properties for triples of finite von Neumann algebras $B\subset N\subset M$, and we introduce families of triples of groups $H<K<G$ whose associated von Neumann algebras $L(H)\subset L(K)\subset L(G)$ satisfy $\mathcal{N}_{L(G)}(L(H))''=L(K)$. It turns out that the latter equality is implied by two conditions: the equality $\mathcal{N}_G(H)=K$ and the above mentioned mixing properties. Our families of examples also allow us to exhibit examples of pairs $H<G$ such that $L(\mathcal{N}_G(H))\not=\mathcal{N}_{L(G)}(L(H))''$.
\end{abstract}

\maketitle

\section{Weakly and strongly mixing finite von Neumann algebras}

The main purpose of the present paper is to present families of triples of groups $H\triangleleft K<G$ whose associated von Neumann algebras have mixing properties  which imply that the $L(K)$ is the von Neumann algebra generated by the normalizer of $L(H)$ in $L(G)$. Thus this section is devoted to the discussion of mixing properties for arbitrary finite von Neumann algebras.
\par
Let $1\in B\subset N\subset M$ be finite von Neumann algebras being endowed with a normal, finite, faithful, normalized trace $\tau$. The normalizer of $B$ in $M$ is denoted by $\mathcal{N}_M(B)$, and it is the group of all unitary elements $u\in U(M)$ such that $uBu^*=B$.
We assume for simplicity that $M$ has separable predual. We denote by $\E_B$ (resp. $\E_N$) the trace-preserving conditional expectation from $M$ onto $B$ (resp. $N$), and we set
$M\ominus N=\{x\in M:\E_N(x)=0\}$. For $v\in U(B)$, let $\sigma_v\in \mathrm{Aut}(M)$ be defined by $\sigma_v(x)=vxv^*$. This gives an action of $U(B)$ on $M$ that preserves the subspace $M\ominus N$. If $G$ is a group, every element $x$ of the associated von Neumann algebra 
$L(G)$
admits a Fourier series decomposition
$x=\sum_{g\in G}x(g)\lm_g$ where $x(g)=\tau(x\lm_{g}^{-1})$ and $\sum_g|x(g)|^2=\Vert x\Vert_2^2$. If $H$ is a subgroup of $G$, then $L(H)$ identifies to the von Neumann subalgebra of $L(G)$ formed by all elements $y$ such that $y(g)=0$ for every $g\in G\setminus H$. The corresponding conditional expectation $\E_{L(H)}$ satisfies $\E_{L(H)}(x)=\sum_{h\in H}x(h)\lm_h$ for every $x\in L(G)$.
\par\vspace{3mm} 
Our first definition is an extension of Definitions 2.1 and 3.4 of \cite{JS} to the case of triples as above.  

\begin{definition}
Consider a triple $1\in B\subset N\subset M$ of finite von Neumann algebras as above and the action $\sigma$ of $U(B)$ on $M$ by conjugation.
\begin{enumerate}
\item [(1)] We say that $B$ is \textbf{weakly mixing in $M$ relative to} $N$ if, for every finite 
set $F\subset M\ominus N$ and every
$\varepsilon>0$, one can find $v\in U(B)$ such that
$$
\Vert \E_{B}(x\sigma_v(y))\Vert_2=\Vert \E_{B}(xvy)\Vert_2<\varepsilon\quad\forall x,y\in F.
$$
If $B=N$, we say that $B$ is \textbf{weakly mixing} in $M$. 
\item [(2)] If $B$ is diffuse, we say that $B$ is \textbf{strongly mixing in} $M$ \textbf{relative to} $N$ if 
$$
\lim_{n\to\infty}\Vert \E_{B}(xu_ny)\Vert_2=0
$$
for all $x,y\in M\ominus N$ and all sequences $(u_n)\subset U(B)$ which converge to 0 in the weak operator topology.
If $B=N$, we say that $B$ is \textbf{strongly mixing in} $M$.
\end{enumerate}
\end{definition}

The next definition introduces a relative version of the so-called \textit{weak asymptotic homomorphism property}; the latter was used first by Robertson, Sinclair and Smith in \cite{RSS} for MASA's in order to get an easily verifiable criterion for singularity. It was proved next by Sinclair, Smith, White and Wiggins in \cite{SSWW} that, conversely, any singular MASA has the weak asymptotic homomorphism property. The relative version of the above property was introduced by Chifan in \cite{Chi} in order to prove that if $A$ is a masa in a separable type $\textrm{II}_1$ factor $M$ then the triple
$$
A\subset \mathcal{N}_M(A)''\subset M
$$
has the relative weak asymptotic homomorphism property. He used it to prove that, if $(M_i)_{i\geq 1}$ is a sequence of finite von Neumann algebras and if $(A_i)_{i\geq 1}$ is a sequence such that $A_i\subset M_i$ is a MASA for every $i$, then
$$
\overline{\bigotimes}_{i\geq 1}\mathcal{N}_{M_i}(A_i)''=(\mathcal{N}_{\overline{\bigotimes}_i M_i}(\overline{\bigotimes}_i A_i))''.
$$
This relative property is related to one-sided quasi-normalizers, as it was proved by Fang, Gao and Smith in \cite{FGS}. See Theorem 1.4 below.

\begin{definition}
Let $1\in B\subset N\subset M$ be a triple of finite von Neumann algebras as in Definition 1.1 above.
\begin{enumerate}
	\item [(1)] The triple of algebras $1\in B\subset N\subset M$ has the \textbf{relative weak asymptotic homomorphism property} if there exists a net of unitaries $(u_i)_{i\in I}$ in $B$ such that
$$
\lim_{i\in I}\Vert \E_B(xu_iy)-\E_B(\E_N(x)u_i\E_N(y))\Vert_2=0
$$
for all $x,y\in M$.
\item [(2)] The \textbf{one-sided quasi-normalizer} of $B$ in $M$ is the set of all elements $x\in M$ for which there exists a finite set $\{x_1,\ldots,x_n\}\subset M$ such that
$$
Bx\subset \sum_{i=1}^n x_iB.
$$
Following \cite{FGS}, we denote the set of these elements by $q\mathcal{N}_{M}^{(1)}(B)$.
\end{enumerate}
\end{definition}

\begin{remark}
(1) If $B\subset N\subset M$ and if $B$ is diffuse and strongly mixing in $M$ relative to $N$, then it is obviously weakly mixing in $M$ relative to $N$.
\par\noindent
(2) If $B$ is strongly mixing in $M$ relative to $N$, then every diffuse von Neumann algebra $1\in D\subset B$ is also strongly mixing in $M$ relative to $N$.
\par\noindent
(3) The following identity 
$$
\E_B(xuy)-\E_B(\E_N(x)u\E_N(y))=\E_B([x-\E_N(x)]u[y-\E_N(y)]),
$$
which holds for every $u\in U(B)$ and all $x,y\in M$, implies that 
the relative weak mixing property and the relative weak asymptotic homomorphism property are equivalent.
\end{remark}

The following theorem is the main result of \cite{FGS}.

\begin{theorem} 
(J. Fang, M. Gao and R. R. Smith)
Let $1\in B\subset N\subset M$ be a triple of finite von Neumann algebras with separable predual. Then the following conditions are equivalent:
\begin{enumerate}
	\item [(1)] The triple $B\subset N\subset M$ has the relative weak asymptotic homomorphism property, or, equivalently, $B$ is weakly mixing in $M$ relative to $N$. 
	\item [(2)] The one-sided quasi-normalizer $q\mathcal{N}_{M}^{(1)}(B)$ is contained in $N$.
\end{enumerate}
\end{theorem}

The next technical result is inspired by Proposition 4.1 of \cite{Popa83} and Lemma 2.1 of \cite{RSS}; it reminds also heredity properties of relative weak mixing from \cite{FGS}.

\begin{proposition}
Let $B\subset N\subset M$ be a triple as above.
\begin{enumerate}
\item [(1)] Assume that $B$ is weakly mixing in $M$ relative to $N$. Then for every nonzero projection $e\in B$, the reduced algebra $eBe$ is weakly mixing in $eMe$ relative to $eNe$. Moreover, 
one has for every $u\in U(M)$:
$$
\Vert \E_{B}-\E_{uBu^*}\Vert_{\infty,2}\geq \Vert u-\E_{N}(u)\Vert_2.
$$
\item [(2)] If $B$ is strongly mixing in $M$ relative to $N$, then for every diffuse unital von Neumann subalgebra $D$ of $B$, one has for every $u\in U(M)$:
$$
\Vert \E_D-\E_{uDu^*}\Vert_{\infty,2}\geq \Vert u-\E_{N}(u)\Vert_2.
$$
\item [(3)] If $B_i\subset N_i\subset M_i$ are triples of finite von Neumann algebras, $i=1,2$, and if $B_i$ is weakly mixing in $M_i$ relative to $N_i$, then $B_1\overline{\otimes}B_2$ is weakly mixing in $M_1\overline{\otimes}M_2$ relative to $N_1\overline{\otimes}N_2$.
\end{enumerate}
\end{proposition}
\begin{proof} The first part of claim (1) is a straightforward consequence of Corollary 6.3 of \cite{FGS} and claim (3) follows from Proposition 6.1 of the same article.
\par
We prove claim (2) because the proof of the last assertion in (1) is similar to that of statement (2). 
\par
Thus fix a diffuse von Neumann subalgebra $D$ of $B$, 
a unitary operator $u\in U(M)$, and let us consider $x=u^*-\E_{N}(u^*)$ and $y=u-\E_{N}(u)\in M\ominus N$ and let $\varepsilon>0$. By the above remark, one has for every $v\in U(D)$:
$$
\E_{D}(xvy)  = 
\E_{D}(u^*vu)-\E_{D}(\E_{N}(u^*)v\E_{N}(u))
$$
As $D$ is diffuse, there exists a sequence of unitaries $(v_n)\subset U(D)$ which converges to 0 with respect to the weak operator topology.
Since $B$ is strongly mixing in $M$ relative to $N$, there exists a positive integer $n$ such that $\Vert \E_{B}(xv_ny)\Vert_2\leq \varepsilon$. As $\E_D=\E_D\E_{B}$, we have 
$\Vert \E_{D}(xv_ny)\Vert_2\leq \varepsilon$ as well.
The above computations give
$$
\Vert \E_{D}(u^*v_nu)\Vert_2\leq \Vert \E_{D}(\E_{N}(u^*)v_n\E_{N}(u))\Vert +\varepsilon
\leq \Vert \E_{N}(u)\Vert_2 +\varepsilon.
$$
We get then:
\begin{eqnarray*}
\Vert \E_{D}-\E_{uDu^*}\Vert_{\infty,2}^2 & \geq &
\Vert v_n-\E_{uDu^*}(v_n)\Vert_2^2\\
&=& \Vert u^*v_nu-\E_{D}(u^*v_nu)\Vert_2^2\\
&=&
1-\Vert \E_{D}(u^*v_nu)\Vert_2^2\\
&\geq &
1-(\Vert \E_{N}(u)\Vert_2+\varepsilon)^2\\
&=&
1-\Vert \E_{N}(u)\Vert_2^2-2\varepsilon\Vert \E_{N}(u)\Vert_2-\varepsilon^2\\
&=&
\Vert u-\E_{N}(u)\Vert_2^2-2\varepsilon\Vert \E_{N}(u)\Vert_2-\varepsilon^2.
\end{eqnarray*} 
As $\varepsilon$ is arbitrary, we get the conclusion.
\end{proof}

We end the present section with a first class of examples of relative strongly mixing algebras; its proof is inspired by that of Lemma 2.2 in \cite{DSS}.

\begin{proposition}
Let $1\in B\subset N$ and $Q$ be arbitrary finite von Neumann algebras with separable preduals. Assume moreover that $B$ is diffuse.
Then $B$ is strongly mixing relative to $N$ in the free product algebra $M=N*Q$.
\end{proposition}
\begin{proof} Let us recall that the free product $N*Q$ is the von Neumann algebra generated by the unital $*$-algebra
$$
P:=\C1\oplus \bigoplus_{n\geq 1}\left(\bigoplus_{i_1\not=\cdots\not=i_n}M_{i_1}^0\otimes \cdots
\otimes M_{i_n}^0\right)
$$
where $M_1^0=\{x\in N:\tau(x)=0\}$ and $M_2^0=\{x\in Q:\tau(x)=0\}$. (We have chosen and fixed finite, normal, faithful normalized traces on $N$ and $Q$.) Thus every element of $P$ is a finite linear combination of words $w$ of the following form: either $w\in N$, or $w$ is a finite product of letters of zero trace and at least one letter belongs to $Q$. 
\par
Then fix a sequence $(u_n)\subset U(B)$ which converges weakly to zero and two words $x,y\in P$ as above. If $x,y\in N$, then $\E_{B}(xu_ny)-\E_{B}(\E_{N}(x)u_n\E_{N}(y))=0$. If $x,y\in Q^0$, then $x,y\in M\ominus N$ and
$xu_ny$ decomposes as a sum
$xu_ny=x(u_n-\tau(u_n))y+\tau(u_n)xy$ where the first term is a reduced word, hence 
$\E_{B}(x(u_n-\tau(u_n))y)=0$, and 
$$
\Vert \E_{B}(xu_ny)\Vert_2\leq |\tau(u_n)|\Vert xy\Vert_2\to 0
$$
as $n\to\infty$. If $x=x_1b_1$ and $y=b_2y_2$ where $b_1,b_2\in N$ and $x_1$ ends with an element in $Q^0$ and $y_2$ starts with an element in $Q^0$, then
\begin{eqnarray*}
\E_{B}(xu_ny)-\E_{B}(\E_{N}(x)u_n\E_{N}(y)) & = &
\E_{B}(x_1b_1u_nb_2y_2)\\
& &
-\E_{B}(\E_{N}(x_1)b_1u_nb_2\E_{N}(y_2))\\
& = &
\E_{B}(x_1\{b_1u_nb_2-\tau(b_1u_nb_2)\}y_2)\\
&  &
+\tau(b_1u_nb_2)\E_{B}(x_1y_2)\\
&= & \tau(b_1u_nb_2)\E_{B}(x_1y_2).
\end{eqnarray*} 
As the words $x_1\{b_1u_nb_2-\tau(b_2u_nb_1)\}y_2$, $x_1$ and $y_2$ are reduced, they belong to the kernels of the conditional expectations $\E_{B}$ and $\E_{N}$. Hence we get
$$
\Vert \E_{B}(xu_ny)-\E_{B}(\E_{N}(x)u_n\E_{N}(y))\Vert_2\leq |\tau(b_1u_nb_2)|\Vert x_1y_2\Vert
\to 0
$$
as $n\to\infty$. The remaining case, when exactly one of $x$ and $y^*$ ends with a letter from $Q$, is dealt with as in the previous case.
\end{proof}

\section{The case of group algebras}

As will be seen below, the associated von Neumann algebras of suitable triples of groups $H<K<G$ give rise to examples of von Neumann algebras which satisfy the relative weak or strong mixing properties.
Our next definition generalizes the so-called \emph{conditions (SS)} and \emph{(ST)} of \cite{JS} to not necessarily abelian groups. We are indebted to the referee for having suggested the following more intuitive formulation. 

\begin{definition}
Let $G$ be a countable group and let $H<K$ be two infinite subgroups of $G$.
\begin{enumerate}
\item [(a)] We say that the triple $H<K<G$ satisfies \textbf{condition (SS)} if all orbits of the natural action $H\curvearrowright (G\setminus K)/H$ are infinite.
\item [(b)] The triple $H<K<G$ satisfies \textbf{condition (ST)} if all stabilizers of the natural action $H\curvearrowright (G\setminus K)/H$ are finite.
\end{enumerate}
When $H<G$ we say that the pair $H<G$ satisfies condition (SS) (resp. (ST)) if it is the case for $H=K<G$.
\end{definition}

We present below technical characterizations of both properties. In order to do that,
recall on the one hand that all orbits of $H\curvearrowright (G\setminus K)/H$ 
are infinite if and only if, for every finite subset $Y\subset G\setminus K$, one can find $h\in H$ such that $hY\cap Y=\emptyset$ (see for instance Lemma 2.2 in \cite{KT}).
\par
On the other hand, for $g,h\in G\setminus K$, set 
$$
E(g,h)=\{\gamma\in H: g\gamma h\in H\}=g^{-1}Hh^{-1}\cap H
$$
and $E(g)=E(g^{-1},g)=gHg^{-1}\cap H$. The latter is a subgroup of $G$. Then it is easy to see that, for arbitrary $\gamma_0\in E(g,h)$, one has $E(g,h)\subset E(g^{-1})\gamma_0$. 
\par 
Making use of these observations, the proof of the following lemma is straightforward.

\begin{lemma}
Let $H<K<G$ be three infinite, countable groups. 
\begin{enumerate}
\item[(a)] The triple $H<K<G$ satisfies condition (SS) if and only if for every nonempty finite set $F\subset G\setminus K$, there exists $h\in H$ such that $FhF\cap H=\emptyset$.
\item[(b)] The following conditions are equivalent:
\begin{enumerate}
\item [(1)] $H<K<G$ satisfies condition (ST);
\item [(2)] for every finite set $F\subset G\setminus K$, there exists an exceptional finite set $E\subset H$ such that $FhF\cap H=\emptyset$ for every $h\in H\setminus E$.
\item [(3)] $E(g,h)$ is a finite set for all $g,h\in G\setminus K$;
\item [(4)] $E(g)$ is a finite group for every $g\in G\setminus K$.
\end{enumerate}
\end{enumerate}
\end{lemma}

Using the same type of arguments as in \cite{JS}, one proves the following generalization of Proposition 2.3 and of Theorem 3.5 in \cite{JS}:

\begin{proposition}
Let $H<K<G$ be a triple of countable groups. Then:
\begin{enumerate}
\item [(a)] The triple $H<K<G$ satisfies condition (SS) if and only if $L(H)$ is weakly mixing in $L(G)$ relative to $L(K)$. 
\item [(b)] The triple $H<K<G$ satisfies condition (ST) if and only if $L(H)$ is strongly mixing in $L(G)$ relative to $L(K)$.
\end{enumerate}
Moreover, if $H<K<G$ satisfies condition (SS) (resp. (ST)) and if $\sigma:G\rightarrow \mathrm{Aut}(Q,\tau)$ is a trace-preserving action of $G$ on some finite von Neumann algebra $(Q,\tau)$, then the crossed product algebra $Q\rtimes_\sigma H$ is weakly (resp. strongly) mixing in $Q\rtimes_\sigma G$ relative to $Q\rtimes_\sigma K$.
\end{proposition}

\begin{remark} (1) An infinite subgroup $H$ of a group $G$ is called \textit{malnormal} if, for every $g\in G\setminus H$, one has $H\cap gHg^{-1}=\{1\}$. Hence such a pair $H<G$ satisfies condition (ST). More generally, $H$ is said to be
\textit{almost malnormal} if, for every $g\in G\setminus H$, the subgroup $H\cap gHg^{-1}$ is finite. Thus, if $H<K<G$ is a triple that satisfies condition (ST), one can say equivalently that $H$ is \textit{almost malnormal in $G$ relative to $K$.}
\newline
(2) Following \cite{V10}, we say that a subgroup $H$ of a group $G$ is \emph{relatively malnormal} if there exists an intermediate subgroup $K<G$ of infinite index such that $gHg^{-1}\cap H$ is finite for all $g\in G\setminus K$. This means precisely that the triple $H<K<G$ satisfies condition (ST).
\end{remark}

If $G$ is a group and $H$ is a subgroup of $G$, there is a straightforward analogue of the one-sided quasi-normalizer of a subalgebra: we denote by $q\mathcal{N}_G^{(1)}(H)$ the set of elements $g\in G$ for which there exists finitely many elements $g_1,\ldots,g_n\in G$ such that $Hg\subset \bigcup_{i=1}^n g_iH$. In view of Theorem 1.3, it is natural to ask whether the triple of algebras $L(H)\subset L(K)\subset L(G)$ has the relative asymptotic homomorphism property if and only if $q\mathcal{N}_G^{(1)}(H)\subset K$. In the special case $H=K$, the authors of \cite{FGS} proved that it is indeed true, but their proof relied heavily on the main result of the article (\emph{i.e.} Theorem 1.3 in the present notes). We succeeded in providing a self contained proof in the context of group algebras in \cite{JWAHP}, and we recall some characterizations in the next theorem (where $\lm$ denotes the left regular representation of $G$ on $\ell^2(G)$.)

\begin{theorem}
(\cite{JWAHP}, Theorem 2.1)
Let $H<K<G$ be a triple of groups and let $B=L(H)\subset N=L(K)\subset M=L(G)$ be their associated von Neumann algebras. Then the following conditions are equivalent:
\begin{enumerate}
	\item [(1)] There exists a net $(h_i)_{i\in I}\subset H$ such that, for all $x,y\in M\ominus N$, one has
	$$
	\lim_{i\in I}\Vert \E_B(x\lm_{h_i} y)\Vert_2=0,
$$
i.e. the net of unitaries in the relative weak asymptotic homomorphism property may be chosen in the subgroup $\lm(H)$ of $U(B)$.
	\item [(2)] The triple $B\subset N\subset M$ has the relative weak asymptotic homomorphism property.
	\item [(3)] The subspace of $H$-fixed vectors $\ell^2(G/H)^H$ in the quasi-regular representation $\bmod\ H$ is contained in $\ell^2(K/H)$.
	\item [(4)] The one-sided quasi-normalizer $q\mathcal{N}_G^{(1)}(H)$ is contained in $K$.
	\item [(5)] The triple $H<K<G$ satisfies condition (SS), i.e. for every non empty finite set $F\subset G\setminus K$, there exists $h\in H$ such that $FhF\cap H=\emptyset$.
\end{enumerate}
\end{theorem}

\begin{remark}
(1) Let $G$ be a group and let $H$ be a subgroup of $G$. Then it is interesting to note
the following description of the one-sided quasi-normalizer of $H$ in $G$; it is certainly known to specialists: 
$$
 q\mathcal{N}_G^{(1)}(H)=\{g\in G: [H:H\cap gHg^{-1}]<\infty\}.
$$
Indeed, let $g\in G$ be arbitrary; let $u_1,\ldots,u_n,\ldots\in H$ be such that 
$$
H=\bigsqcup_j u_j(H\cap gHg^{-1}).
$$
Then it is easy to see that 
$$
HgHg^{-1}=\bigsqcup_j u_jgHg^{-1}.
$$
It implies that $HgH=\bigsqcup_j u_jgH$, hence that $[H:H\cap gHg^{-1}]<\infty$ if and only if $g\in q\mathcal{N}_G^{(1)}(H)$.
\\
(2) Let $H<K<G$ be a triple of groups. Condition (SS) is equivalent to the following apparently weaker \textit{condition (wSS)}: 
\par
\textit{For every nonempty finite set $F\subset G\setminus K$ and for every $g\in G\setminus K$, there exists $h\in H$ such that $Fhg\cap H=\emptyset$.} 
\par
Indeed, it is easy to see that condition (wSS) is a reformulation of condition (4) of Theorem 2.5, namely that $q\mathcal{N}_G^{(1)}(H)\subset K$.
\end{remark} 

We use Theorem 2.5 to give examples of triples of algebras $B\subset N\subset M$ with $N=\mathcal{N}_M(B)''$ in the case of group von Neumann algebras and crossed products.

\begin{theorem}
Let $H<K<G$ be a triple of groups such that $K=\mathcal{N}_G(H)$ is the normalizer of $H$ in $G$. 
We denote by $W\sp\ast(q\mathcal{N}_{L(G)}^{(1)}(L(H)))$ the von Neumann algebra generated by the one-sided quasi-normalizer of $L(H)$ in $L(G)$.
Then the following two conditions are equivalent:
\begin{enumerate}
	\item [(1)] $L(K)=W\sp\ast(q\mathcal{N}_{L(G)}^{(1)}(L(H)))$;
	\item [(2)] $H<K<G$ satisfies condition (SS).
\end{enumerate}
Thus, if the triple $H<K<G$ satisfies condition (SS), then $\Nor_{L(G)}(L(H))''=L(K)$,
and, moreover,
if $G$ acts on some finite von Neumann algebra $Q$, then 
$$
Q\rtimes K=\mathcal{N}_{Q\rtimes G}(Q\rtimes H)''.
$$
\end{theorem}
\begin{proof} 
(1) $\Rightarrow$ (2). By hypothesis, the triple of algebras $L(H)\subset L(K)\subset L(G)$ has the relative weak asymptotic homomorphism property, hence, by Theorem 2.5, $H<K<G$ satisfies condition (SS).
\newline
(2) $\Rightarrow$ (1).
By assumption, one has the following chain of inclusions and equalities:
$$
K=\mathcal{N}_G(H)\subset q\mathcal{N}_{G}^{(1)}(H)\subset K
$$
where the first inclusion follows from the definition and the second one from condition (SS). At the level of von Neumann algebras, this gives:
$$
L(K)=L(\mathcal{N}_G(H))\subset \mathcal{N}_{L(G)}(L(H))''\subset W^*(q\mathcal{N}_{L(G)}^{(1)}(L(H)))\subset L(K)
$$
where the last inclusion follows from the relative weak asymptotic homomorphism property. 
\par
For the case of crossed products,
the unitary groups $U(Q)$ and $K$ are contained in $\Nor_{Q\rtimes G}(Q\rtimes H)$, hence $Q\rtimes K\subset \Nor_{Q\rtimes G}(Q\rtimes H)''$.
As the latter algebra is trivially contained in $W^*(q\Nor_{Q\rtimes G}^{(1)}(Q\rtimes H))$, the assertion follows from Proposition 2.3 above, and from Theorem 3.1 of \cite{FGS}.
\end{proof}

\begin{remark}
(1) Let $H<G$ be a pair of groups; one may ask whether it is possible to have $\Nor_{L(G)}(L(H))''\not=W^*(q\mathcal{N}_{L(G)}^{(1)}(L(H)))$. It is indeed the case, as explained below.
\par 
Following Section 5 of \cite{FGS}, we denote by $H_1$ the quasi-normalizer of $H$ in $G$, \textit{i.e.}, the maximal subgroup of $q\Nor_G^{(1)}(H)\cap q\Nor_G^{(1)}(H)^{-1}$, and we let $H_2$ denote the subgroup of $G$ generated by $q\Nor_G^{(1)}(H)$. 
Finally, let $q\Nor_{L(G)}(L(H))''$ be the quasi-normalizer algebra of $L(H)$ in $L(G)$, \textit{i.e.}, the two-sided version of the von Neumann subalgebra $W^*(q\mathcal{N}_{L(G)}^{(1)}(L(H)))$.
Then Corollary 5.2 of \cite{FGS} shows that 
$$
q\Nor_{L(G)}(L(H))''=L(H_1)
$$ 
and that 
$$
W\sp\ast(q\mathcal{N}_{L(G)}^{(1)}(L(H)))=L(H_2).
$$ 
Example 5.3 of the same article shows that it may happen that $H_1\not=H_2$, hence that it is possible to have a triple $H<K=H_2<G$ that satisfies condition (SS) (by Corollaries 4.1 and 5.2 of \cite{FGS}) but with 
$\Nor_{L(G)}(L(H))''\subsetneqq W^*(q\mathcal{N}_{L(G)}^{(1)}(L(H)))$.
\newline
(2) Let $H<G$ be a pair of groups as above, $H$ being abelian. Corollary 5.7 of \cite{FGS} shows that if, furthermore, $L(H)$ is a MASA in $L(G)$, then 
$$
\mathcal{N}_{L(G)}(L(H))''=L(\mathcal{N}_G(H)).
$$
As will be seen in the next section, the last equality can fail if we do not assume that $L(H)$ is a MASA in $L(G)$. 
\par 
In fact,
the family of examples of triples of groups in the next section allows us to propose examples as well as counterexamples to the above equality, namely, we exhibit triples $H<K<G$ such that $K=\mathcal{N}_G(H)$ and $L(K)=\Nor_G(L(H))''$ on the one hand, and, on the other hand, we will see that there are triples $H<K<G$ with $K=\mathcal{N}_G(H)$
but $L(K)\subsetneqq \mathcal{N}_{L(G)}(L(H))''$.
\end{remark}

We end the present section with a result that applies in the framework of condition (ST), but its relationship with strong mixing is still unclear. It is a generalization of Corollary 2.6 of \cite{Popa83}. In order to state it, we recall the definition of commuting squares.
\par 
Let $M$ be a finite von Neumann algebra endowed with a normal, finite, faithful, normalized trace $\tau$ and let $B_0$ and $B_1$ be von Neumann subalgebras of $M$ with the same unit. The diagram
$$
\begin{array}{ccc}
B_0 & \subset & M\\
\cup & & \cup\\
B_0\cap B_1 & \subset & B_1
\end{array}
$$
is a \textit{commuting square} if 
$$
\E_{B_0\cap B_1}(b_0b_1)=\E_{B_0\cap B_1}(b_0)\E_{B_0\cap B_1}(b_1)
$$ 
for all $b_j\in B_j$, $j=0,1$, or, equivalently, if $\E_{B_0}\E_{B_1}=\E_{B_1}\E_{B_0}=\E_{B_0\cap B_1}$. See Chapter 4 in \cite{GHJ} for other equivalent conditions and further details.
As is well known, if $G_1$ and $G_2$ are subgroups of $G$, the system of inclusions 
$$
\begin{array}{ccc}
L(G_1)& \subset & L(G)\\
\cup & & \cup\\
L(G_1\cap G_2) & \subset & L(G_2)
\end{array}
$$
is a commuting square. Thus, for every subgroup $H$ of $G$ and for every $g\in G$, the following diagram is a commuting square
$$
\begin{array}{ccc}
L(gHg^{-1})& \subset & L(G)\\
\cup & & \cup\\
L(gHg^{-1}\cap H) & \subset & L(H).
\end{array}
$$
In particular, if $H<K<G$ is a triple that satisfies condition (ST), if $g\in G\setminus K$, the commuting square
$$
\begin{array}{ccc}
L(gHg^{-1})& \subset & L(G)\\
\cup & & \cup\\
L(gHg^{-1}\cap H) & \subset & L(H)
\end{array}
$$
satisfies the hypotheses of the following proposition since $L(gHg^{-1}\cap H)$ is finite-dimensional, hence atomic for every $g\in G\setminus K$. Thus, every such $g$ is orthogonal to $\Nor_{L(G)}(L(H))''$.

\begin{proposition}
Let $M$ be a finite von Neumann algebra endowed with some normal, faithful, finite, normalized trace $\tau$, let $1\in B\subset M$  be a diffuse von Neumann subalgebra of $M$ and let $u\in U(M)$ be such that
$$
\begin{array}{ccc}
uBu^*& \subset & M\\
\cup & & \cup\\
uBu^*\cap B & \subset & B
\end{array}
$$
is a commuting square, and assume that $uBu^*\cap B$ has the following properties: its center is atomic and its relative commutant $(uBu^*\cap B)'\cap M$ is diffuse (the latter conditions are automatically satisfied if $uBu^*\cap B$ itself is atomic). Then $u$ is orthogonal to $\mathcal{N}_M(B)''$.
\end{proposition}
\begin{proof} 
Let us first fix $v\in\mathcal{N}_M(B)$ and let us prove that $\tau(vu)=0$. As $vBv^*=B$, the diagram
$$
\begin{array}{ccc}
vuBu^*v^*& \subset & M\\
\cup & & \cup\\
vuBu^*v^*\cap B & \subset & B
\end{array}
$$
is a commuting square as well, $C:=vuBu^*v^*\cap B=v(uBu^*\cap B)v^*$ has atomic center and its relative commutant $C'\cap M$ is still diffuse.
Moreover, recall that one has $\E_C(xb)=\E_C(bx)$ for all $b\in C'\cap B$, $x\in M$, and that $\E_C(b)\in Z(C)$. 
\par
If $(z_j)_{j\geq 1}$ is the set of minimal projections of the center $Z(C)$, then each reduced algebra $Cz_j$ is a finite subfactor of the reduced von Neumann algebra $z_jMz_j$, and its relative commutant is the diffuse algebra $z_j(C'\cap B)z_j=(Cz_j)'\cap z_jBz_j$. If $\tau_j$ is the normalized trace on $z_jMz_j$ defined by $\tau_j(z_jxz_j)=\frac{1}{\tau(z_j)}\tau(z_jxz_j)$ for all $x\in M$, then the associated conditional expectation $\E_{Cz_j}$ satisfies the following identity: $\E_{Cz_j}(z_jxz_j)=\E_C(z_jxz_j)z_j$ for all $x\in M$. In particular, one has $\E_{Cz_j}(y)=\tau_j(y)z_j$ for every $y\in z_j(C'\cap B)z_j$.
\par
Let us fix $\varepsilon>0$. We claim that there exists a partition of unity $(e_i)_{1\leq i\leq n}$ in $C'\cap B$ so that $\Vert \E_C(e_i)\Vert\leq\varepsilon$ for every $i$. Indeed, choose a positive integer $m$ such that $2^{-m}\leq\varepsilon$ and then, for every $j$, a partition of the unity $(e_{j,i})_{1\leq i\leq 2^m}\subset z_j(C'\cap B)z_j$ such that $\tau_j(e_{j,i})=2^{-m}$ for all $i$. Finally, set $n=2^m$ and $e_i=\sum_je_{j,i}$. Then $(e_i)_{1\leq i\leq n}$ is a partition of the unity in $C'\cap B$ and, by the above considerations,
$$
\E_C(e_i)=\sum_j\tau_j(e_{j,i})z_j=2^{-m}\leq\varepsilon\quad\forall i.
$$
Let $D$ be the abelian von Neumann algebra generated by the projections $(e_i)$. We recall that 
$$
\E_{D'\cap M}(x)=\sum_i e_ixe_i\quad\forall x\in M,
$$
and we will make use of the following identity
$$
\E_C(vue_iu^*v^*e_i)=\E_C(vue_iu^*v^*)\E_C(e_i)
$$
which is true since $vue_iu^*v^*\in vuBu^*v^*$, $e_i\in B$ and by the commuting square condition. One has:
\begin{eqnarray*}
|\tau(vu)|^2 &=&
|\tau(\E_{D'\cap M}(vu))|^2\leq \Vert \E_{D'\cap M}(vu)\Vert_2^2\\
&= &
\tau(|\E_{D'\cap M}(vu)|^2)=
\tau(\E_C(|\E_{D'\cap M}(vu)|^2)).
\end{eqnarray*}
But,
\begin{eqnarray*}
\E_C(|\E_{D'\cap M}(vu)|^2) &=&
\E_C\left(\left|\sum_i e_ivue_i\right|^2\right)\\
&=&
\E_C(\sum_{i,j}e_ivue_ie_ju^*v^*e_j)\\
&=&
\sum_i \E_C(e_ivue_iu^*v^*e_i)\\
&=&\sum_i \E_C(vue_iu^*v^*)\E_C(e_i)\\
&=&
\sum_i \E_C(vue_iu^*v^*)^{1/2}\E_C(e_i)\E_C(vue_iu^*v^*)^{1/2}\\
&\leq&
\varepsilon \sum_i\E_C(vue_iu^*v^*)=\varepsilon \E_C(vuu^*v^*)=\varepsilon.
\end{eqnarray*}
Thus, one has $|\tau(vu)|^2\leq \varepsilon$. By weak density, we get $\tau(xu)=0$ for every $x\in\mathcal{N}_M(B)''$.
\end{proof}

\begin{remark}
Let $1\in C\subset M$ be a pair of finite von Neumann algebras. As we have seen in the proof above, the existence, for every $\varepsilon>0$ of partitions of the unity $(e_i)_{1\leq i\leq n}$ in $C'\cap M$ such that $\Vert \E_C(e_i)\Vert \leq \varepsilon$ for all $i$, is implied by two conditions: (i) the center $Z(C)$ is atomic, and (ii) the relative commutant $C'\cap M$ is diffuse. In fact, the existence of such partitions of the unity requires both conditions. Indeed, suppose for instance that $M$ is a type II$_1$ factor; firstly, if $C$ is a MASA in $M$, then its center and its relative commutant are equal and diffuse, and $\E_C(e)=e$ for every nonzero projection $e\in C'\cap M=C$, thus $\Vert \E_C(e)\Vert=1$ for all such projections. Secondly, if $C$ is a subfactor of finite index in $M$, then its center is atomic, but its relative commutant is finite dimensional. Hence there exists a constant $c>0$ such that
$$
\Vert E_C(e)\Vert =\tau(e)\geq c
$$
for every nonzero projection $e\in C'\cap M$.
\end{remark}

\section{Families of examples}

As indicated above, we are going to give examples of triples of groups which satisfy conditions (SS) or (ST). They will be based on semidirect products.
\par 
Thus, let $K$ be a group, let $H<K$ be a subgroup of $K$ and assume that $K$ acts on some group $A$ through an action $\al$. Put $G=A\rtimes K$, and identify $K$ with the subgroup $\{e\}\times K$ of $G$. For future use, we put $A^*=A\setminus\{e\}$.
\par 
A case that might be of interest comes from generalized wreath product groups as in the next example.

\begin{example}
Let $K$ be an arbitrary countable group. Assume that $K$ acts on some countable set $X$, and take any nontrivial group $Z$. Let $A=Z^{(X)}$ be the group of all maps $a:X\rightarrow Z$ such that $a(x)=e$ except for some finite subset of $X$. Then $K$ acts by left translation on $A$, and the corresponding group $G$ is the generalized wreath product group $Z\wr_X K$. 
\end{example}

We present first a condition which implies that $K$ is the normalizer of $H$ in $G$.

\begin{proposition}
Let $H<K<G=A\rtimes K$ be a triple as above. Assume furthermore that $H$ is a normal subgroup of $K$ and that $e\in A$ is the only element $a\in A$ such that $\al_h(a)=a$ for all $h\in H$. Then $\Nor_G(H)=K$.
\end{proposition}
\begin{proof}
One has obviously $K\subset\Nor_G(H)$. Conversely, if $g=(a,k)\in \Nor_G(H)$, then one has for every $h\in H$:
\begin{eqnarray*}
ghg^{-1}
&=&
(a,k)(e,h)(\al_{k^{-1}}(a^{-1}),k^{-1})=(a,kh)(\al_{k^{-1}}(a^{-1}),k^{-1})\\
&=&
(a\al_{khk^{-1}}(a^{-1}),khk^{-1})
\end{eqnarray*}
which belongs to $H$ if and only if $\al_{khk^{-1}}(a)=a$ for every $h\in H$, if and only if $\al_h(a)=a$ for every $h\in H$. Hence $a=e$.
\end{proof}

Let us see now under which conditions the triple $H<K<G=A\rtimes K$ satisfies condition (SS) or (ST). It is partly inspired by Theorem 2.2 of \cite{RSS}. See also Proposition 4.6 of \cite{JS}.

\begin{theorem}
Let $H<K<G=A\rtimes K$ be a triple as above.
\begin{enumerate}
\item[(a)] The triple $H<K<G$ satisfies condition (ST) if and only if, for every $a\in A^*$, $|\{h\in H:\al_h(a)=a\}|<\infty$.
\item[(b)] The triple $H<K<G$ satisfies condition (SS) if and only if, for every finite set $E\subset A^*$, there exists $h\in H$ such that $E\cap \al_h(E)=\emptyset$. Equivalently, for every $a\in A^*$, the $H$-orbit $H\cdot a$ is infinite.
\end{enumerate}
\end{theorem}
\begin{proof}
(a) $(\Rightarrow)$ Let $a\in A^*$. Then $(a,e)\in G\setminus K$, and $H\cap (a,e)H(a^{-1},e)$ is finite. But, if $h\in H$, one has
$(a,e)(e,h)(a^{-1},e)=(a\al_h(a^{-1}),h))\in H$ if and only if $\al_h(a)=a$. The set of such elements $h\in H$ must be finite.\\
$(\Leftarrow)$ Let $g\in G\setminus K$; let us prove that $H\cap gHg^{-1}$ is finite. Put $g=(a,k)$. If $h\in H$, one has
\begin{eqnarray*}
ghg^{-1}
&=&
(a,k)(e,h)(\al_{k^{-1}}(a^{-1}),k^{-1})\\
&=&
(a,kh)(\al_{k^{-1}}(a^{-1}),k^{-1})\\
&=&
(a\al_{khk^{-1}}(a^{-1}),khk^{-1}).
\end{eqnarray*}
Thus, $ghg^{-1}\in H$ if and only if $khk^{-1}\in H$ and $\al_{khk^{-1}}(a)=a$. Hence $H\cap gHg^{-1}$ is finite.
\par\vspace{1mm}\noindent
(b) $(\Rightarrow)$ Let $E\subset A^*$ be finite. Replacing $E$ by $E\cup E^{-1}$ if necessary, we assume that $E=E^{-1}$. Then $F:=\{(a,e): a\in E \}\subset G\setminus K$ and there exists $h\in H$ such that $F(e,h)F\cap H=\emptyset$. In particular, $(e,h)(a,e)=(\al_h(a),h)\notin FH$ for every $a\in E$ and $E\cap\al_h(E)=\emptyset$.\\
$(\Leftarrow)$ Let $F\subset G\setminus K$ be finite. There exist $F_1\subset A^*$ and $F_2\subset K$ finite such that $F\subset F_1\times F_2$. Observe that, for $(a_j,k_j)\in F_1\times F_2$, $j=1,2$, and for $h\in H$, one has
$$
(a_1,k_1)(e,h)(a_2,k_2)=(a_1,k_1h)(a_2,k_2)=(a_1\al_{k_1h}(a_2),k_1hk_2)
$$
hence, if it belongs to $H$, one must have $a_1\al_{k_1}(\al_h(a_2))=1$, or equivalently, $\al_h(a_2)=\al_{k_1^{-1}}(a_1^{-1})$. Taking 
$E=F_1\cup\{\al_{k^{-1}}(a^{-1}): k\in F_2, a\in F_1 \}$, if $h\in H$ is such that $E\cap \al_h(E)=\emptyset$, then necessarily, $F(e,h)F\cap H=\emptyset$.
\end{proof}

\begin{remark}
Let us consider the case where $A=Z^{(X)}$ as in Example 3.1: $K$ acts on the countable set $X$. Then $|\{h\in H:\al_h(a)=a\}|<\infty$ for every $a\in A^*$
if and only if, for every $x\in X$, the stabilizer $H_x=\{h\in H: h\cdot x=x\}$ is finite. By Proposition 2.3 of \cite{KT}, it is equivalent to the fact that the action of $H$ by generalized Bernoulli shifts on $(Y^X,\nu^X)$ is mixing, where $(Y,\nu)$ is some standard probability space. 
\par
In the same vein, for every nonempty, finite set $E\subset A^*$ one can find $h\in H$ such that $E\cap \al_h(E)=\emptyset$ if and only if the action of $H$ on $X$ has infinite orbits. By Proposition 2.1 of \cite{KT}, it is equivalent to the fact that the corresponding Bernoulli shift action of $H$ on $(Y^X,\nu^X)$ is weakly mixing.
\end{remark}

The following corollary is a straightforward consequence of Theorem 2.7, Proposition 3.2 and of Theorem 3.3. 
 
\begin{corollary}
Let $H, K$ and $A$ be as above, denote by $H<K<G$ the associated triple of groups and assume that:
\begin{enumerate}
	\item [(1)] for every $a\in A^*$, there exists $h\in H$ such that $\al_h(a)\not=a$;
	\item [(2)] for every finite set $E\subset A^*$, there exists $h\in H$ such that $E\cap \al_h(E)=\emptyset$.
\end{enumerate}
Then $L(K)=\mathcal{N}_{L(G)}(L(H))''$. More generally, if $G$ acts on some finite von Neumann algebra $(Q,\tau)$, then 
$Q\rtimes K=\mathcal{N}_{Q\rtimes G}(Q\rtimes H)''$.
\end{corollary}

As promised in the preceeding section, we give examples of triples $H<K<G$ such that $\mathcal{N}_G(H)=K$ but $L(K)\subsetneqq \mathcal{N}_{L(G)}(L(H))''$.

\begin{example}
Let $H\triangleleft K$ be infinite, countable groups, let $A$ be an infinite, countable group endowed with an action of $K$, and assume that:
\begin{enumerate}
	\item [(i)] for every $a\in A^*$, there exists $h\in H$ such that $\al_h(a)\not=a$;
	\item [(ii)] there exists $a_0\in A^*$ whose orbit $H\cdot a_0$ is finite.
\end{enumerate}
Let $H<K<G$ be the associated triple of groups. Then the first condition implies that $\mathcal{N}_G(H)=K$ by Proposition 3.2, and the second one implies that the triple does not satisfy condition (SS) by Theorem 3.3. Let us prove that $L(K)\subsetneqq \mathcal{N}_{L(G)}(L(H))''$. Put
$F=\{(\al_h(a_0),e): h\in H\}$
which is a finite set by the second condition above. Define
$$
x=\sum_{g\in F\cup F^{-1}}\lm_g.
$$
Then $x$ is a selfadjoint element of $L(G)\ominus L(K)$ since $F\cap K=\emptyset$. Furthermore, it is easy to see that $x\in L(H)'\cap L(G)$. As $x\notin L(K)$, there exists a spectral projection $e$ of $x$ which does not belong to $L(K)$ either. Then $u:=2e-1$ is a unitary element of $L(H)'\cap L(G)$ hence it belongs to the normalizer $\mathcal{N}_{L(G)}(L(H))$, but $u\notin L(K)$. Thus $L(K)\subsetneqq \mathcal{N}_{L(G)}(L(H))''$. Observe that $H$ can be an abelian group, and this shows that, in order to have the equality 
$\mathcal{N}_{L(G)}(L(H))''=L(\mathcal{N}_G(H))$ in Corollary 5.7 of \cite{FGS}, one must assume that $L(H)$ is a MASA in $L(G)$.
\end{example}

\bibliographystyle{amsplain}
\bibliography{max}

\end{document}